\documentclass[12pt]{article}
\usepackage{amssymb}
\usepackage{amsmath,amsthm}

\providecommand{\U}[1]{\protect\rule{.1in}{.1in}}
\makeatletter
\def\@seccntformat#1{\csname the#1\endcsname.\quad}
\makeatother

\newtheorem{theorem}{Theorem}

\newtheorem{corollary}{Corollary}

\newtheorem{definition}{\noindent Definition}
\newtheorem{example}{Example}

\newtheorem{lemma}{Lemma}

\newtheorem{proposition}{Proposition}
\newtheorem{remark}{Remark}

\renewenvironment{proof}[1][Proof]{\noindent\textbf{#1.} }{\ \rule{0.5em}{0.5em}}
\begin{document}

\title{{\Large \textbf{On a new proof of the Prime Number Theorem }}}
\author{ Dhurjati Prasad Datta\thanks{
email:dp${_-}$datta@yahoo.com} \\
Department of Mathematics, University of North Bengal \\
Siliguri,West Bengal, Pin: 734013, India }
\date{}
\maketitle

\begin{abstract}
A new elementary proof of the prime number theorem  presented recently in the framework of a scale invariant extension of the ordinary analysis is re-examined and clarified further. Both the formalism and proof are presented in a much more simplified manner. Basic properties of some key concepts such as infinitesimals, the associated nonarchimedean absolute values, invariance of measure and cardinality of a  compact subset of the real line under an IFS are discussed more thoroughly. Some interesting applications of the formalism in analytic number theory are also presented. The error term as dictated by the Riemann hypothesis also follows naturally thus leading to an indirect proof of the hypothesis.
\end{abstract}


\textbf{MSC Numbers: 26A12, 11A41}

{\bf Key Words:} Scale free, Nonarchimedean, Infinitesimals, Prime number theorem
\newpage
\baselineskip=15.5pt

\section{Introduction}

We present a simple proof of the Prime Number Theorem (PNT) \cite{rh} on the basis of
a more refined treatment of the elementary concept of limit in the ordinary
analysis. This work is in continuation of our current studies on the
formulation of a scale invariant analysis that aims at developing a coherent
framework for analysis on the real line $R$ as well as Cantor like fractal
subsets of $R$ \cite{dp1,dp2}. The present proof is a more simplified version of the proof
presented recently \cite{dp3}. By Cantor sets we mean any compact, perfect, no-where 
dense subset of $R$. 
Although the results presented in this work should be valid for any such  general
Cantor set, we, for the sake of simplicity, restrict ourselves
to a class of homogeneous Cantor sets defined as the limit set of an IFS of
the form

\begin{equation}  \label{ifs}
f_1(t)=\beta t, \ f_2(t)=\beta t+ (1-\beta)
\end{equation}
\noindent where $0<\beta<1$ and $\alpha+2\beta=1$. Both the box and
Hausdorff dimensions of the limit set are known to be $s=\frac{\log 2}{\log
\beta^{-1}}$. The precise nature of effects coming from  a choice of the Cantor set will be studied 
in a more detailed analysis elsewhere.

To state the PNT, let $\Pi (x)=\underset{p<x}{\sum }1$, $p$ being the $p$th
prime, be the prime counting function, i.e., $\Pi (x)$ counts the number of
primes less than a given real number $x$. Then the PNT states that $\Pi (x)/(%
{\frac{x}{\log x}})=1,$ as $x\rightarrow \infty $. The problem of
determining the precise form of the correction term of this asymptotic formula
is still unsettled. According to the Riemann hypothesis (RH), $\Pi (x)/({
\frac{x}{\log x}})=1+O(x^{-{\frac{1}{2}}+\epsilon }),$ as $x\rightarrow
\infty $, for any $\epsilon >0$. Although all the experimental searches on
primes are known to agree with this RH correction, no analytical proof is still 
available  in the literature \cite{rh}. In our approach, we study $\Pi(x^{-1})$, when $x \rightarrow 0^+$, so that the correction term, in consonance with the RH, has the form  $\Pi (x^{-1})\times {x}{\log x^{-1}}=1+O(x^{{\frac{1}{2}}-\epsilon }),$ as $x\rightarrow
0^+ $, for any $\epsilon >0$. It is heartening to note that the present approach appears to produce the above RH correction term {\em naturally}, thus leading to a proof of the celebrated Riemann Hypothesis, though in an indirect way. The classical Hadamard- de la Vall$\acute{e}$e Poussin proof of the PNT vis-$\acute{a}$-vis the present approach will be examined separately.

\vspace{.5cm}

\section{Scale Invariant formalism:}

The formulation of a {\em scale invariant analysis} was motivated by our effort in justifying the construction of the so-called {\em nonsmooth} solutions \cite{dm} of the simplest scale invariant Cauchy problem

\begin{equation}\label{sfe}
t{\frac{dX}{dt}}=X, \ X(1)=1
\end{equation}

\noindent in a rigorous manner. It is clear that the framework of classical analysis, because of Picard's uniqueness theorem, can not rigorously accommodate such solutions, except possibly only in an approximate sense. To bypass the obstacle, it became imperative to look for a non-archimedean extension of the classical setting, thus allowing for existence of nontrivial {\em infinitesimals} (and hence, by inversion, infinities). Robinson's original models of nonstandard analysis (or any minor variations of the same) appeared to be unsatisfactory, because (i) infinitesimals here are {\em infinitesimals} even in ``values", (ii) the value of an infinitesimal is the usual Euclidean value and (iii) these are {\em new extraneous elements} in $R$ (a comparison with Nelson's approach \cite{nels} with the present one will be considered elsewhere). Although, the nonstandard $\bf R$ {\em is} non-archimedean, but still an infinitesimal behaves more in a``real number like" manner; that is to say, in essence, it fails to have an {\em identity}, except for its infinitesimal Euclidean value. Such nonstandard infinitesimals are known to generate proofs of harder theorems of mathematical analysis in a more intuitively appealling  manner (see, for example, \cite{tao}). Further, any new theorem proved in the nonstandard approach is expected to have a classical analysis proof, though, may be, using lengthier arguments. Justifying a higher derivative discontinuous (nonsmooth) solution of (\ref{sfe}), therefore, appeared to be difficult even in the conventional nonstandard analysis. 

To counter this problem, we contemplated developing a {\em novel} non-archimedean extension of $R$ by completing the rational number field $Q$ under a {\em novel} ultrametric which treats arbitrarily small and large rational  numbers {\em seperately} from finite, {\em moderately large} rational (and real)  numbers \cite{dp3}. The ultrametric reduces to the usual Euclidean value for finite real numbers, but, nevertheless, leads to a {\em new} definition (realization) of scale invariant infinitesimals in the present context. An important feature of this formalism is that we make use of traditional $\epsilon-\delta$ techniques of classical analysis, but applied instead on a {\em deformed real number system } $\mathcal R$, when the deformation is induced  by the ultrametric valuation of scale invariant relative infinitesimals (Definition 1,2). The straightforward derivation of the PNT in $\mathcal R$ clearly reveals the strength of this formalism, even at this relatively initial stage of its development.

We introduce scale invariant infinitesimals via a more refined evaluation of the limit $x\rightarrow 0^+$ in ${R}$. Notice that as $x\rightarrow 0^+$, there exists $\delta>0$ such that $0<\delta<x$ and one usually identifies zero (0) with the closed interval $I_{\delta}=[-\delta,\delta]$ at the {\em scale} (i.e. accuracy level in a computation) $\delta$. Ordinarily, $I_{\delta}$ is a connected line segment, which shrinks to the singleton $\{0\}$ as $\delta\rightarrow 0^+$, so as to reproduce the infinitely accurate, {\em exactly determinable} ordinary real numbers. This also tells that in the usual (classical) sense there is no room for a scale when one talks about a limit of the form $x\rightarrow 0$.

Let us now present yet another {\em nontrivial mode} realizing the limiting motion $x\rightarrow 0^+$. This mode is nonlinear, as it is defined via an {\em inversion law}, rather than simply by linear translations, that is available {\em uniquely} to a real variable to undergo changes in the standard analysis  on $R$. In the presence of {\em infinitesimals}, the present nonlinear mode is shown to gain significance. In the following we give a definition and also present a nontrivial construction for an {\em explicitly defined infinitesimal}, without requiring the  set up of Robinson's nonstandard analysis. As remarked already,  our formalism offers a new, independent realization of inifinitesimals residing originally in an ultrametric space, but, nevertheless, inducing nontrivial influences in the form of an {\em real valued ``infinitesimal correction"} of the form $\delta(x)\propto x\log x^{-1}$ to an arbitrarily small real variable $x$ approaching 0, thus making the existence of infinitesimals analytically (and dynamically as well) more {\em meaningful and significant}. As a consequence, a real number $x$, as it were, is raised to   {\em deformed values} $X_{\pm} \in \mathcal R$ (Sec. 3), because of a possible nontrivial {\em motion} close to 0. This construction also reveals the exact sense how a {\em scale} might become relevant in the present scale invariant formalism.

\begin{definition}\cite{dp1,dp2,dp3,dp4,dp5}
Let $x\in I=[-1,1]\subset R$ and $x$ be arbitrarily small, i.e., $x\neq 0$, but, nevertheless, $x\rightarrow 0^+$. Then there exists $\delta>0$ and a set of (positive) \emph{relative infinitesimals} $\tilde x$ relative to the scale $\delta$ satisfying $
0<\tilde x<\delta<x $ and the \emph{inversion} rule $\tilde x/\delta\propto
\delta/x$. The associated scale invariant infinitesimals are defined by $
\tilde X=\underset{\delta\rightarrow 0^+}{\lim} \tilde x/\delta$.
\end{definition}

\begin{example}
Let $x_n=\epsilon^{n(1-l)}, \ 0<l<1, \ 0<\epsilon<1$. 
Then scale invariant relative infinitesimals are $\tilde
X_{n\lambda}=\lambda\epsilon^{nl}, \ 0<\lambda<1$, when $\delta=\epsilon^n$,
for a sufficiently large $n$, is chosen as a scale. Analogously, for a continuous 
variable $x$ approaching $0^+$, say, and considered as a scale, a class of the relative 
infinitesimals are represented as $\tilde x = x^{1+l}(1+o(x)), \ 0<l<1$, so that the corresponding scale invariant 
infinitesimals are defined by the asymptotic formula $\tilde X=\lambda x^l +o(x^m), \ m>l$. Notice that a scale
invariant infinitesimal goes to zero at a smaller (ultrametric) rate $l$: $
\tilde X=\lambda x^l \ \Rightarrow \ {\frac{d\log \tilde X}{d\log x}=l}$.
\end{example}

\begin{remark}{\rm 
In the limit $\delta \rightarrow 0$, the set of relative infinitesimals
apparently reduces to the singleton $\{0\}$. However, the set of scale
invariant infinitesimals is nevertheless nontrivial, in the sense that these will 
have nontrivial influences on the real number system. To uncover the nature of such 
influences, the set of infinitesimals, denoted $\mathbf{0}$, is now awarded with a \emph{natural}
weight (absolute value) with nontrivial effects on the number theory and
other areas of analysis. The set $\mathbf{0}$ has the {\em asymptotic
representation} $\mathbf{0}=\{0, \delta \tilde X\}$, as $ \delta\rightarrow 0^+$.
For definiteness, the ordinary zero (0) is called the \emph{stiff} zero, when
the nontrivial infinitesimals are called \emph{soft} zeros. The ordinary
real line $R$ is then extended over $\mathbf{R}=\{\mathbf{\tilde x}: \mathbf{%
\tilde x}=x+\mathbf{0}, \ x\in R\}$, which as a field extension, and as a consequence of the Frobenius theorem, must be an
infinite dimensional nonarchimedean space. To re-emphasize, our aim here is to show nontriviality of soft zeros in obtaining an easy proof of the PNT.}
\end{remark}

\begin{definition} \cite{dp1,dp2,dp3,dp4,dp5}
An absolute value $v:\mathbf{0}\rightarrow I^+=[0,1]$
\begin{equation}
v(\tilde x)\equiv ||\tilde x|| := \underset{\delta\rightarrow 0^+}{\lim}%
\log_{\delta^{-1}} {\tilde X}^{-1}, \ \tilde x=\delta \tilde X \ \in \mathbf{%
0}
\end{equation}
\noindent together with $v(0)=0$ is assigned to the set of infinitesimals $%
\mathbf{0}$. Infinitesimals weighted with above absolute value are called
\emph{valued} infinitesimals.
\end{definition}

Let us list a few more interesting features of scale invariant infinitesimals and the associated absolute value \cite{dp3, dp4, dp5}. 

{\bf Remark 2}: 

1. As remarked already, the set of infinitesimals ${\bf 0} =\{0\}$ when $\epsilon\rightarrow 0$ classically. However, the corresponding asymptotic expressions for the scale free (invariant) infinitesimals  are nontrivial, in the sense that the associated valuations (Definition 2) can be shown to exist as {\em finite real numbers}. Below we give a definite construction indicating the exact sense how relative infinitesimals and associated values arise in a limiting problem. 

Fix a value of $\delta=\delta_0$ and let $C_{\delta_0}\subset [0, \delta_0]\equiv I_{\delta_0}^+$ be a Cantor set defined by an IFS of the form 

\begin{equation}\label{ifs1}
f_1(x)=\lambda x, \ f_2(x)=\lambda x-(\lambda/\delta_0 -1)\delta_0
\end{equation}
\noindent where $\lambda=\beta\delta_0, \ 0<\beta<1$ and $\alpha+2\beta=1$. Thus, at the first iteration an open interval $O_{11}$ of size $\alpha\delta_0$ is removed from the interval $I_{\delta_0}^+$, at the second iteration 2 open intervals $O_{21}$ and $O_{22}$  each of  size $\alpha\delta_0(\beta)$ are removed and so on, so that a family of gaps $O_{ij}$ of sizes $\alpha\delta_0(\beta)^{i-1}, \ j=1,\ldots, 2^{i-1}$  are removed in subsequent iterations from each of the closed subintervals $I_{ij}, \ j=1,2,\ldots, 2^{i} $ of $I_{\delta_0}^+$. Consequently, $C_{\delta_0}=I_{\delta_0}^+ \ -\ \underset{i}{\cup}O_{ij}=\underset{i}{\cap}\underset{j}{\cup}I_{ij}$. Notice that the total length removed is $\sum \alpha\delta_0(2\beta)^{i-1}=\delta_0$, so that the linear Lebesgue measure $m(C_{\delta_0})=0$.

Next, consider $\tilde I_{N}=[0, \beta^N]$ and let $N=n+r$ and $N\rightarrow \infty$ as $n\rightarrow \infty$ for a fixed $r\geq 0$. Choose the scale $\delta=\alpha \beta^n\delta_0$ and define $\tilde x_r\in [0, \alpha\beta^N\delta_0]$ a ( positive ) {\em relative infinitesimal} (relative to the scale $\delta$) provided it also satisfies the {\em inversion} rule $\tilde x/\delta = \lambda \delta/x$ (c.f. Definition 1), for a real constant $\lambda(\delta) $ $(0\ll\lambda \leq 1)$. For each choice of $x$ and $\delta$, we have a unique $\tilde x$ for a given $\lambda \in (0,1)$. Consequently, by varying $\lambda$ in an open subinterval of (0,1), we get an open interval of relative infinitesimals in the interval $(0,\delta)$, all of which are related to $x$ by the inversion formula. In the limit $\delta\rightarrow 0$, relative infinitesimals $\tilde x_r$, of course, vanish identically. However, the corresponding scale invariant infinitesimals ${\tilde X}_r=\tilde x_r/\delta, \ \delta \rightarrow 0$ are, nevertheless,  nontrivial and weighted with  new {\em scale invariant absolute values} via Definition 2.

 The set of infinitesimals are uncountable, and as shown below the above norm satisfies the stronger triangle inequality $v(x+y)\leq {\rm max}\{v(x),v(y)\}$. Accordingly, the zero set ${\bf 0}= \{0, \pm\delta {\tilde X}_r| \ {\tilde X}_r\in (0,\beta^r), \ r=0,1,2,\ldots, \ \delta \rightarrow 0^+\}$ may be said to acquire {\em dynamically} the structure of a Cantor like ultrametric space, for each $\beta\in (0,1/2)$ (so as to satisfy the open set condition \cite {dp3}). The set $\bf 0$ indeed is realized as a set of nested circles $S_r: \{\tilde x| \ v(\tilde x_r)=\alpha_r\}$, in the ultrametric norm, when we order, without any loss of generality, $\alpha_0>\alpha_1>\ldots$. The ordinary 0 of $R$ is  replaced by this set of scale free infinitesimals $0\rightarrow \ \bar {\bf 0} ={\bf 0}/\sim = \{0, \cup S_r\}$; $\bar{\bf 0}$ being the equivalence class under the equivalence relation $\sim$, where $x \sim y$ means $v(x)=v(y)$ . The existence of $\tilde x$ could also be concieved dynamically as a computational model \cite{dp1,dp2,dp3}, in which a number, for instance, 0 is identified as an interval $[-\delta, \delta]$ at an accuracy level determined by  $\delta=\beta^n$.

2. The concept of infinitesimals and the associated absolute value considered here become significant only in a limiting problem (or process), which is reflected in the explicit presence of ``$\underset{\epsilon\rightarrow 0}{\lim}$" in the relevant definitions. Recall that for the continuous real valued function $f(x)=x$, the statement $\underset{x\rightarrow 0}{\lim} \ x=0$, means that ${x\rightarrow 0}$ essentially {\em is} $x=0$. This may be considered to be a {\em passive} evaluation (interpretation) of limit. The present approach is {\em active (dynamic)}, in the sense that it offers not only a more refined evaluation of the limit, but also provides a clue how one may induce new (nonlinear) structures (ingredients) in the limiting (asymptotic) process. The inversion rule (Definition 1) is one such {\em nonlinear structure} which may act nontrivially as one investigates more carefully the {\em motion} of a real variable $x$ (and hence of the associated scale $\delta < x$) as it goes to 0 more and more accurately. Notice that at any ``instant",  elements defined by inequalities $0<\tilde x<\delta<x$ in a limiting process, are well defined; relative infinitesimals are {\em meaningful} only in that {\em dynamic} sense (classically, these are all zero, as $x$ itself {\em is } zero). Scale invariant infinitesimals $\tilde X$, however, may or may not be zero classically. $\tilde X=\mu \ (\neq 0), $ a constant,  for instance, is nonzero even when $x$ and $\epsilon$ go to zero. On the other hand, $\tilde X= \delta^{\alpha}, \ 0<\alpha<1$, of course, vanish classically, but as shown below, are nontrivial in the present formalism. As a consequence, relative (and scale invariant) infinitesimals may said to {\em exist} even as real numbers in this dynamic sense. The acompaning metric $||\cdot ||$, however, is an ultrametric.

3. However, a genuine (nontrivial) scale free  infinitesimal $\tilde X$  can not be a constant. Let $\tilde x_0=\mu\delta, \ 0<\mu <1$, $\mu$ being a constant. Then $v(\tilde x_0)=\underset{\delta\rightarrow 0}{\lim} \log _{\epsilon}\mu =0$, so that $\tilde x_0$ is essentially the trivial infinitesimal 0 (more precisely, such a relative infinitesimal belongs to the equivalence class of 0).

4. The scale free infinitesimals of the form $\tilde X_m\approx \delta^{\alpha_m}+ o(\delta^{\beta}), \ \beta>\alpha$ go to 0 at a slower rate compared to the linear motion of the scale $\delta$. The associated nontrivial  absolute value $v(\tilde x_m)=\alpha_m$ essentially quantifies this {\em decelerated} motion.

\begin{theorem}
v has following properties.

(i) $v$ is an ultrametric, and hence $\mathbf{0}$ equipped with $v$ is an
ultrametric space.

(ii) $v$ is a locally constant Cantor function.
\end{theorem}

\begin{proof}
 (i) (a) $v$ is well defined. Indeed, the open set $\bf 0$ (in the  usual topology, for each fixed $\delta$) is written as a countable union of disjoint open intervals $I_{\delta i}$ of relative infinitesimals  ${\bf 0}=\bigcup I_{\delta i}$. Let $v(\tilde x_i)=\alpha_i$, a constant  for all $\tilde x_i (= \lambda\delta\delta^{\alpha_i})\in \bar I_{\delta i}$, the closure of $I_{\delta i}$. Thus $v$ exists and well defined.

(b) Let $0<\tilde x_2<\tilde x_1<\tilde x_1+\tilde x_2<\delta$ be two relative infinitesimals. We have, $0<\tilde X_2<\tilde X_1<\tilde X_1+\tilde X_2<1$ and $v({\tilde x}_2)>v({\tilde x}_1)>v(\tilde x_1+\tilde x_2)$, thus proving the strong triangle inequality $v(\tilde x_1+\tilde x_2)\leqq {\rm max}\{v(\tilde x_1),v(\tilde x_2)\}$.

Next, given $0<\tilde x<\delta$, there exist a constant $0<\sigma(\delta)<1$ and $a: {\bf 0}\rightarrow R$, such that $\tilde X=\lambda\delta^{v(\tilde x)}$ and $v(\tilde x)=\sigma^{a(\tilde x)}$. Accordingly, $a(\tilde x)$ is a discrete valuation satisfying (i) $a(\tilde x_1\tilde x_2)=a(\tilde x_1)+a(\tilde x_2)$, (ii) $a(\tilde x_1+\tilde x_2)\geq {\rm min}\{a(\tilde x_1), a(\tilde x_2)\}$. As a result, $v(\tilde x_1\tilde x_2)=v(\tilde x_1)v(\tilde x_2)$. Hence, $\{{\bf 0}, v\}$ is an ultrametric space.

(ii) Let $\bar{\bf 0}=(\cup \bar I_{\delta i})\bigcup (\cup J_k)$, the closure of $\bf 0$. The open intervals $J_k$ are gaps between two consecutive closed intervals $\bar I_{\delta i}$. $J_k$'s actually contain new points those arise as the limit points of sequences of the end points of the open intervals $I_{\delta i}$. Clearly, $\bar{\bf 0}$ is connected in usual topology. However, in the ultrametric topology, both $I_{\delta i}$ and $J_k$ are clopen sets and $\bar{\bf 0}$ is totally disconnected. Since, it is bounded and also is perfect $\bf 0$ is equivalent to an ultrametric Cantor set.

Now, the local constancy of $v$ in the ultrametric $\bar{\bf 0}$ follows from the definition:

\begin{equation}\label{lc}
{\frac{d v(\tilde x)}{dx}} \ =\ \underset{\delta\rightarrow 0^+}\lim {\frac{d}{dx}}\left({\frac{\log  x}{\log\delta}} +1\right) \ =\ 0
\end{equation}

The vanishing  derivative above arises from a logarithmic divergence  arising from the  nontrivial finer scales. This is unlike the ordinary analysis, when one interprets $\bar{\bf 0}$ as a connected subset of $R$, thereby forcing $v$ to vanish uniquely, so as to recover the usual structure of $R$. The above vanishing derivative can be interpreted nontrivially as a LCF when $x\in R$ is supposed to belong to a Cantor subset of $I$ \cite{dp4}.

Eq(\ref{lc}) also reveals the {\em reparaterization invariance} of a locally constant valuation $v(x)$. As a consequence, $v$ may be a function of any reparametrized monotonic variable $\tilde x=\tilde x(x)$ with $\tilde x^{\prime}(x)>0$, instead being simply a function of the original real variable $x$.

\begin{remark}{\rm 
The differentiability in the metric space $\bf 0$ is defined as usual : $f: {\bf 0}\rightarrow R$ is differentiable at $x=a\in \bf 0$ with derivative $f^{\prime}(a)$ if $\underset{x\rightarrow a}{\lim} {\frac{|f(x)-f(a)|}{||x-a||}}=f^{\prime}(a)$, when $x\rightarrow a$ means $||x-a||\rightarrow 0^+$. In the above derivation, however, derivative is evaluated in usual metric which is equivalent  to  the natural ultrametric of the Cantor set inhabiting $x$. This also tells that differentiability in a Cantor set is well defined even in natural metric when increments in such a set is defined by inversions rather than by linear shifts \cite{dp1,dp2}.}
\end{remark}

Now to construct a general class of locally constant functions in the ultrametric space, let us proceed as in (ia) above, with the supposition that the constants $\alpha_i$'s are arranged in ascending order. Thus, $v(\tilde x_i)=\alpha_i, \ \alpha_i\leq\alpha_j \Leftrightarrow\ i\leq j$ for all $\tilde x_i\in I_i$ (we drop the suffix $\delta$ for simplicity). Clearly, (\ref{lc}) holds over for all $I_i$. On the other hand, for an $\tilde x\in J_k$, where $J_k$ separates two consecutive $I_i$ and $I_{i+1}$, say, so that $\tilde x_i<\tilde x<\tilde x_{i+1}$, where $\tilde x_i$ is the right end point of $I_i$ and $\tilde x_{i+1}$ is the left endpoint of $I_{i+1}$,   we have $v(\tilde x_{i+1})- v(\tilde x_i)=(\alpha_{i+1}-\alpha_i)$. Because of ultrametricity, one can always choose $\alpha_i=\beta_{ij_i}\sigma(i)^s$, for $\beta_{ij_i}>0$ ascending and $\sigma(i) \rightarrow 0$ as $i\rightarrow \infty$ and $j_i=0,1,2,\ldots k(i)$ for some $i$ dependent constant   $k(i)$ (c.f. 2nd para of (1b) ). Consequently, $v(\tilde x_{i+1})- v(\tilde x_i)=(\beta_{(i+1)j_{i+1}}- \beta_{ij_i})\sigma(i)^s$. It follows that the sequence $v(\tilde x_{i+1})$ is increasing and $v(\tilde x_i)$ is decreasing. Thus, $v(\tilde x):= \lim v(\tilde x_i)$ as $i\rightarrow \infty$. Hence, $v: {\bf 0} \rightarrow I^+$ is indeed a Cantor function.

Conversely, given a Cantor function $\phi(x), \ x\in I^+$, one can define a class of infinitesimals $\tilde x\approx \delta\delta^{\phi(\tilde x/\delta)}$ belonging to the extended set $\bf 0$ for $\delta\rightarrow 0^+$. This completes the proof.
\end{proof}

\begin{remark}{\rm 
To re-emphasise, the valuation $v$ can be considered to quantify the degree of {\em decelerated} motion as the real variable $x\rightarrow 0$ because of obstructions offered
by nontrivial scale invariant infinitesimals in the ultrametric Cantor set of
$\mathbf{0}$. The usual Euclidean norm is a measure of a finite real number
because of its position relative to 0.}
\end{remark}

\begin{definition}
Besides the usual Euclidean value, a real variable $x\neq 0$, but $%
x\rightarrow 0^+$ gets a deformed ultrametric value given by $v( x):=
\underset{\delta\rightarrow 0^+}{\lim}\log_{\delta^{-1}} (x/\delta)$. Then $%
v(x)=v(\tilde x)$.
\end{definition}

\begin{proof}
Because of inversion rule, $x/\delta=\lambda(\delta/\tilde x), \ 0<\lambda<1$%
, and hence $v(x)=v(\tilde x)$ since $\lim \log_{\delta^{-1}}\lambda^{-1}=0 $.
\end{proof}

To proceed further, let us formulate some basic features of $v$ revealing its nature of variability.

\begin{lemma}
Let $0<|x|<|x^{\prime}|$ be two arbitrarily small real variables and $\delta$
be a scale such that $0<\delta<|x-x^{\prime}| <|x|<|x^{\prime}|$. Then $%
v(x^{\prime})=v(x)$.
\end{lemma}

\begin{proof}
From definition 3, $v(x-x^{\prime})<v(x)<v(x^{\prime})$. But $x^{\prime}=x+(x^{\prime}-x)$. So by ultrametric inequality, $v(x^{\prime})\leq {\rm max}\{v(x), v(x^{\prime}-x)\}\leq v(x)$.
\end{proof}

\begin{lemma}
Let $0<|x|<|x^{\prime}|$ be two arbitrarily small real variables and $\delta$
and $\delta^{\prime}$ be two scales such that $0<\delta<|x|<\delta^{%
\prime}<|x^{\prime}|$. The corresponding scale invariant infinitesimals are $%
\tilde X$ and $\tilde X^{\prime}$ with associated valuations $v(x)$ and $%
v(x^{\prime})$. Then $v(x^{\prime})=(\alpha/s)v(x)$, where $\alpha=\lim
\log_{\tilde X} \tilde X^{\prime}$, determines the gap size between $\tilde
X $ and $\tilde X^{\prime}$ and $s=\lim \frac{\log \delta^{\prime}}{\log
\delta}$ is the Hausdorff dimension of the Cantor set of infinitesimals as $%
x, \ x^{\prime}\rightarrow 0$.
\end{lemma}

\begin{proof}
The proof follows from

\begin{equation}
{\frac{v(x^{\prime})}{v(x)}}=\lim {\frac{\log (x^{\prime}/\delta^{\prime})}{\log (x/\delta)}}\times \lim {\frac{\log {\delta}}{\log \delta^{\prime}}}
\end{equation}

\noindent so that $\alpha=\lim \log_{x/\delta} (x^{\prime}/\delta^{\prime})= \lim \log_{\tilde X} \tilde X^{\prime}$ $\Rightarrow $  $\tilde X^{\prime}=X^{\alpha}(1+O(\beta(x, x^{\prime})), \ \beta\rightarrow 0$ faster than the linear approach
$x\rightarrow 0 $.
\end{proof}

\begin{corollary}
Let $0<\delta<\delta^{\prime}<x$ be two scales in association with an
arbitrarily small real variable and $\tilde X=(x/\delta)^{-1}$ and $\tilde
X^{\prime}=(x/\delta^{\prime})^{-1}$ be the corresponding scale invariant
infinitesimals. Then $v(x^{\prime})=(\alpha/s)v(x)$, where $\alpha=\lim
\log_{\tilde X} \tilde X^{\prime}$, determines the gap size between $\tilde
X $ and $\tilde X^{\prime}$ and $s=\lim \frac{\log \delta}{\log
\delta^{\prime}}$ is the Hausdorff dimension of the Cantor set of infinitesimals as $%
x, \ x^{\prime}\rightarrow 0$.
\end{corollary}

\begin{definition}
A scale invariant jump is defined by the pure inversion $\tilde X^{\prime}=
X^{-1}$ with the scale invariant minimal jump size $\alpha=1$. The jump size
$\alpha$ thus runs over the set of natural numbers $N$.
\end{definition}

\begin{remark}
{\rm Lemma 1 characterizes the equivalence classes of infinitesimals with identical valuations. Subsequent lemma (and corollary) tells that the valuation $v$ changes only when an infinitesimal from one equivalence class switches over to another class.
}
\end{remark}

Summing up the above observations, we now state a general representation of relative infinitesimals and corresponding valuation.

\begin{lemma}\cite{dp4}
A relative infinitesimal $\tilde x$ relative to the scale $\delta$ has the dominant asymptotic form
\begin{equation}\label{rep}
\tilde x= \delta \times \delta^{l} \times \delta^{\phi(\tilde x/\delta)}(1+o(1))
\end{equation}
with associated valuation $v(\tilde x)= l + \phi(\tilde x/\delta)$, where $l\geq 0$ is a constant and $\phi$ is a nontrivial Cantor function.  
\end{lemma}

\begin{proof}
The locally constant $v=v_0 +v_1$ solves ${\frac{dv}{dx}}=0$ and so the above ansatz is the more general solution, with the trivial ultrametric valuation $v_0=l$ and the nontrivial valuation $v_1=\phi$. The representation for $\tilde x$ now follows from definition.
\end{proof}

\begin{remark}{\rm 
As a real variable $x$ and the associated scale $\delta<x$ approach 0, the corresponding infinitesimals $0<\tilde x<\delta$ may live (in contrast to measure zero Cantor sets considered so far) in a positive measure  Cantor set $C_p$, say. Such a possibility is already considered in \cite{dp4} in relation to an interesting phenomenon of {\em growth of measure}. In such a case $v_0(\tilde x)=m(C_p)=l$, the Lebesgue measure of $C_p$. The nontrivial component $v_1$ then relates to the uncertainty (fatness) exponent of the positive measure 1-set. In this extended model (c.f. Sec. 4), the valuation quantifies the presence of {\em nontrivial } motion in a limiting process: $v_0$ gives the uniform scale invariant motion when  $v_1$ arises from the associated nonuniformity stemming out from measure zero Cantor sets. }
\end{remark}

\begin{definition}
The norm $||.||$, induced by $v$, in $\mathbf{R}$ is defined by $||\mathbf{x}%
||=x, \ 0\neq x\in R$, but $||x||=v(x),$ when $x\in \mathbf{0}$.
\end{definition}

\begin{lemma}
$\mathbf{R}$ in the above norm is a locally compact ultrametric space.
Moreover, the topology induced by this norm on punctured set $R-\{0\}$ is
equivalent to the usual topology and the corresponding embedding $i:
R\rightarrow \mathbf{R}$ is continuous.
\end{lemma}

The proof is straightforward and so is not included.

\begin{lemma}
The topology induced by $||.||$ in $\mathbf{0}$ is distinct from the usual
one on the interval (0, $\delta$), $\delta\rightarrow 0^+$.
\end{lemma}

The proof follows from the observation that the a sequence of the form $%
\epsilon^{n-nl}, \ 0<l<1$ converges to 0 in the usual norm for all $l\neq 0$%
, but, nevertheless, converges to $l$ in $||.||$, since $||%
\epsilon^{n-nl}||=v(\epsilon^{n-nl})=l$, when $\delta=\epsilon^n
\rightarrow 0$ as $n\rightarrow \infty$ (c.f. Example 1 and Definition 2).
We note, incidentally, that the sequence $\epsilon^n$ converges to 0 in both
the norms (when the scale $\delta$ is identified with $\epsilon^n$ itself). $\Box$

For other possibilities see Example 2. 

Geometrically, the singleton $\{0\}$ is replaced by an ultrametric Cantor
set $C$ and supposed to be attached `vertically' at 0 in relation to the
`horizontal' real line. The valuation of infinitesimals living in $\mathbf{0}
$ is associated to the valuation of the underlying Cantor set, and depends
on the problem at hand. The class of Cantor sets that may be attached at 0
is rather huge, leading to the possibility of different limiting behaviours
for a given sequence because of different choices of scale
invariant infinitesimals $\tilde X_i$ living in the Cantor set $C_i$.

\begin{example}
Consider the sequence $x_n=2^{-n}$. The standard limit is 0. (i) Choose $%
\delta=p^{-n}$ as a scale for a prime $p>2$ so that $0<\delta<x_n$. Then the(positive) relative infinitesimals are of the form $\tilde x_n=\lambda(p/2)^{-n}p^{-n}, \ 0<\lambda<1$ be a constant and $0<\tilde
x_n<\delta<x_n$. Then $||x_n||= \underset{\delta\rightarrow 0^+}{\lim}%
\log_{\delta^{-1}} {\tilde X_n}^{-1} = 1-\frac{\log 2}{\log p}$. Notice that
infinitesimals $\tilde x_n$ belong to a Cantor set of Hausdorff dimension $%
\frac{\log 2}{\log p}$. Thus depending on the nature/choice of Cantor set
containing infinitesimals, the limit might be different.
\end{example}

Another important  implication of the above observation is considered in the following sections.

To summarize, the ordinary real number system $R$ is extended to a nonarchimedean space $\bf R$ with nontrivial 
valued infinitesimals. More details of the structure of $\bf R$ is available in Sec.5 and in \cite{dp3}. These valued infinitesimals now, in turn, leave an imprint on the real number set so as 
to {\em deform} the original set \cite{dp4}. As become evident the deformed system will be archimedean with the usual topology. 

\vspace{.5cm}

\section{Deformed Real Number System} 

An ordinary real $x$ is  extended over to the fattened
variables $X_{\pm}=x {\mathcal X}_{\pm}, \ {\mathcal X}_{\pm}=x^{\mp v(\tilde x/x)}$ (so that $X_+>x$ and $X_-<x$).
The fattened variables $X_{\pm}$ live in a space $\mathcal R$ called the {\em deformed 
real number set}. Clearly, $R\subset \cal R$, since $v(0)=0$. The deformation is induced by scale invariant infinitesimals $\tilde X$ living in the ultrametric space ${\bf 0}/\{0\}$. Recall that,  scale invariant elements of 
the extended ultrametric space $\mathbf{R}$ in an
infinitesimally small neighbourhood of $x$  belongs to an ultrametric  Cantor set $\bf C$ (homeomorphically equivalent to a $Z_p$) and undergo changes by \emph{inversions} of the form $\tilde
X_{\pm}=\tilde x_{\pm}/x\rightarrow {\tilde X_{\pm}}^{\prime}= \tilde X_{\pm}^{\pm
e^h}, \ h\in \bf 0$ \cite{dp4}. Such ultrametric infinitesimals and the associated inversions and related properties  in $\bf R$, in turn, leave an {\em imprint} on the deformed set $\mathcal R$, so that the real valued deformation factors ${\mathcal X}_{\pm}$ not only live in an associated Cantor set (in usual topology) $\mathcal C \subset R$, in the {\em deformed neighbourhood} of 1,  but also can be thought to undergo changes by inversions of the form ${\mathcal X}_{\pm} \rightarrow {{\mathcal X}_{\pm}}^{\prime}= {\mathcal X}_{\pm}^{\pm e^h}$ where $h$ now is a {\em real parameter}. Notice that there exists a natural homeomorphism between $\bf C$ and ${\mathcal C}\subset {\mathcal R}$. As a consequence, $\mathcal R$  locally has the structure of a positive measure Cantor set. Notice that given a real number $x$ and a scale $\delta$, there exists a continuous mapping $f: \ R\rightarrow R$ such that $x=\delta f(\delta)$. In the deformed system, this mapping $f$ is induced by ultrametric infinitesimals in the form  $f(\delta)=\delta^{-v(\tilde x)}$.

Now, let us recall that differential shift increments in $R$ are defined by $x
\rightarrow x^{\prime}=x+h$ so that $dx=x^{\prime}-x=h, \ h\rightarrow 0$.
In the present case, the inversion induced \emph{jump increments} in $\mathcal R$, over and above the usual shift differential, are defined by $d_jx= \log\log_{\mathcal X}{\mathcal X}^{\prime}= \log\log (\tilde
x_0/\tilde x)$ (we drop the suffix $\pm$ for simplicity), where $\mathcal
X^{\prime}>\mathcal X>1$, say. Indeed, for each $\mathcal X$ which can be interpreted
as a scale ($>$1), there exists a class of `infinitesimals' (actually, `infinities') ${\mathcal X}^{\prime}$ lying in a connected interval so that $\log_{\mathcal X}{\mathcal X}
^{\prime}$ tends to a nonzero constant ($>$1, say)  $\alpha$ as the real variable $
x\rightarrow 0$.  As a consequence, a line segment acquires a fractal like
structure: a singleton $\{x\}$ of the real number system $R$, under the scale invariant action of the said infinitesimals, now gets {\em deformed and fractured} into a Cantor like set. The scale invariant extension of the line segment inhabiting the
variable $\mathcal X$ is the Cantor set $\mathcal C \subset R$ and $\alpha=\log_{\tilde X}{\tilde X^{\prime}} $ gives an estimate of the size of a jump connecting two points $\mathcal X$ and ${\mathcal X}^{\prime}$ in $\mathcal C $ (c.f. Lemma 2). Clearly, the
above estimate can also be written as $\log (\tilde x_0/\tilde x)$, where
a point $\tilde x_0$ of a Cantor set is replaced by a connected segment over
which the real variable $\tilde x$ (with a slight abuse of notation, we are here using the same symbol which denoted infinitesimals in $\bf 0$) is supposed to live in \cite{dp1,dp2}.

In the light of the above remarks, the following two lemmas now succintly encode the various incremental modes in the deformed space $\mathcal R$.

\begin{lemma}
The scale invariant deformed variables ${\mathcal X}$ living in a neighbourhood of 1 in the Cantor like subsets of   ${\mathcal R}/R$ undergo transitions by inversions ${\mathcal X} \ \rightarrow \ {\mathcal X}^{-\alpha}$, where $\alpha>0$ is a multiple of a natural number.
\end{lemma}

\begin{lemma}
Let $x\in R$ and $X\in \mathcal R$. The differential increments for $X\in \mathcal R$ are classified as:  (i) $X^{\prime}=X+h$ for linear shifts taking over $R$, (ii) $X^{\prime} = e^{h^{\prime}}X \ \Rightarrow \ \log X^{\prime}= \log X+ h^{\prime}$, for infinitesimal scaling and (iii) $X^{\prime}=X^{-e^{h^{\prime\prime}}} \ \Rightarrow \ \log\log X^{\prime} = \log\log X ^{-1} + h^{\prime\prime}$, for nonlinear inversions on ${\mathcal R}/R$. As a consequence, the linear and nonlinear increments are related as $h^{\prime}= \log h$ and $h^{\prime\prime} = \log\log h$, where $h, \ h^{\prime}$ and $h^{\prime\prime}$ are sufficiently small real variables.
\end{lemma}

These two lemmas are almost self explanatory in view of the above remarks and Lemma 2. The infinitesimal scaling is an effect of non-zero infinitesimals belonging to a specific equivalence class, having a nonzero {\em constant} valuation $v(\tilde x)=\alpha$, since $X=x\cdot \delta^{v(\tilde x)}=x\cdot \delta^{\alpha}$, over a connected gap.  Nontrivial inversion induced variations are {\em revealed} only under {\em double logarithmic scales} of an ordianry linear variable, when there is a transition from one gap to another (that is to say, between two points of the underlying Cantor set , for more details see \cite{dp2,dp4}).
   
\begin{theorem}
The Cantor function $v(\tilde x)$ solves (\ref{lc}) everywhere in $I^+$,
thereby inducing a smoothening of the ordinary derivative discontinuity at a
point $\tilde x_0$ of the underlying Cantor set $C$, say. The actual
variability of $v$ is reflected at a logarithmic scale $\log (\tilde
x_0/\tilde x)$ in the deformed space $\mathcal R$, i.e.,
\begin{equation}  \label{vari}
\log y{\frac{dv}{d\log y}}=-v, \ y=\tilde x_0/\tilde x
\end{equation}
\noindent when the $\tilde x_0\in \mathcal C$ is replaced by an infinitesimal
connected line segment in which  the variable $\tilde x$ lives.
\end{theorem}

\begin{proof}
The Cantor function $v:I\rightarrow I$ is constant on the gaps of the Cantor set $\mathcal C$. At a point of the Cantor set, the variation of $v$ is given by Lemma 2.
It follows from the fattened variable representation that $v$ solves (\ref{vari}) thus establishing the variability of $v$ in the logarithmic variables, viz, $d\log v=-d\log\log (x_0/\tilde x)$, in the neighbourhood of each $x$ close to 0 when both $x$ and 0 are supposed to be embedded in $\mathcal R$.
\end{proof}

The proof of the PNT is derived when the real variable $x$ is assumed to live in $\mathcal R$. But before delving into this let us explore another interesting feature related to measure and cardinality of a set.
\vspace{.5cm}

\section{Invariance of measure and Cardinality} 

The usual ``cut and delete" process realizing a Cantor set seems to give a misleading
idea that the Lebesgue measure of a set, for instance, [0,1]  becomes zero
under recursive applications of a contraction mapping of the form, say, $%
f(x)=ax,\ 0<a<1$. However, the full Lebesgue measure could actually be
preserved at every level of iterations: $1=a^{n}\times p^{ns}$ for a given $%
p>a^{-1}$, thus leading to a Cantor set of fractal dimension $s=\frac{\log
a^{-1}}{\log p}$, that remains attached to the limit point, viz. 0, of the
said contraction \cite{dp5}. For a finite length $l$, the above balance is achieved as 
$l=(a^{n}l)\times p^{ns}$. 
Notice that the reduction of the original Lebesgue measure
is compensated multiplicatively by the corrective scaling factor $p^{ns}$,
associated with the limiting Cantor set , that is thought to be formed
dynamically as the points in the original set [0,1] are scrambled (jumbled)
together gradually under successive applications of the contraction mapping
so as to degenerate finally into a totally disconnected Cantor set having
the $s$ dimensional Hausdorff measure identically equal to the Lebesgue
measure of the original connected set. More interestingly, the cardinality
of the original set is also preserved. Recall that according to the
classical analysis, the cardinality of the connected set $(0,\delta ),\
\delta =a^{n}$ is the continuum $c$ for each $n$, but the limit set $\{0\}$
has cardinality 0. This singular behaviour remains unexplained in the
standard approach. However, in the present scale invariant framework, the
singularity is removed by a dynamic realization of a Cantor set of cardinality
$c$. Moreover this Cantor set need not have to be considered to be a meager
set, since because of the above realization of the full Lebesgue measure as
the $s$- Hausdorff measure.

The failure of the classical analysis to appreciate the above two facts
appears to stem from its very simplistic treatment of limit. A real variable
$\delta$ may tend to 0 either continuously assuming all the values from the
connected interval $(0,\delta)$, or by jumps from one length say, $\delta[0,1%
]$, to the next smaller length $\delta^2[0,1]$ and so on, via contractions.
These two processes are more or less indistinguishable in the standard
analysis, both yielding the unique limit 0.

However, in real world applications, there are various situations where a
scope for a nontrivial interpretation naturally arise. Consider a bounded
interval of real line containing 0 laid out as an elastic string. Suppose a
ball is rolled so as to approach slowly toward 0, and finally coming to
rest at 0. The limiting concept of classical analysis applies here quite
well. The physical (Euclidean) distance of the ball from 0 vanishes
continuously, giving a satisfactory perception of the actual physical
problem.

But, if we now think that the elastic string itself is contracted
successively (without dissection), even in an idealistic situation, when the
entire string is ultimately contracted to the limiting zero size, then the
usual analytical picture is insufficient! Not a single point of the string is
lost; the whole mass of points of the string is ultimately squeezed together
to form, as it were, a blob, rather than a singleton set. In the light of
the present scale invariant approach, it now transpires that the blob
actually would be a Cantor like set accommodating all the points of the
elastic string of real line in a \emph{scrambled} manner. The original
(initial) length of the string would be transformed in the logarithmic scale
to the Hausdorff measure of the Cantor set.

The present measure invariance appears to be complimentary to the standard
interpretation of Hausdorff measure as the invariance of mass function (or
content): Under an IFS, the total mass content 1 of the set [0,1] gets
uniformly distributed with value $1/r^s$ ($r$ being the scale factor of the
IFS) into $p$ number fragmented set so that $p\times (1/r^s)=1$, so that the
limiting Cantor set so formed has the Hausdorff dimension $\frac{\log p}{%
\log r}$, having the $s$-Hausdorff measure as the invariant mass content of
the set. In the present scenario, the reduced size of the original set is
recovered exponentially as the powers of the scale factor $1/p$ of the
limiting Cantor set for a $p>a^{-1}$. For a $p<a^{-1}$, on the other hand,
one could recast the above measure invariance as $a^{1/s}\times p=1$, thus
realizing the Cantor set with Hausdorff dimension $\tilde s=1/s=\frac{\log p%
}{\log a^{-1}}$.

To explain the above invariance, we note that a length $l<1$ is preserved
multiplicatively by creating a length $L$ at an orthogonal direction so that
the (hyperbolic) inversion rule $lL=1$ holds good. This rule is preserved
even in the limit $l\rightarrow 0$ thus creating a space in the form a
Cantor set in the orthogonal direction. In the following theorem we consider
a homogeneous Cantor set with one gap and two bridges so that $p=2$.

\begin{theorem}
An arbitrarily small closed interval of length $l (\approx 0)$ under the repeated applications of a
contraction $f(x)=ax, \ 0<a<1$ degenerates into a limiting Cantor set $C_l$ of
Hausdorff dimension $s=\frac{\log 2}{\log a^{-1}}, \ 0<a<1/2$ in the
orthogonal direction ($y$ axis) of infinitesimally small elements. The
differential measure of the size of the Cantor set is given by $%
L^{-1}=l^{1/s}$ when $l\rightarrow 0$.
\end{theorem}

\begin {proof}
Under each applications of contraction the reduction of the original length $l \approx 0$ is balanced multiplicatively in the logarithmic scales following the rule

\begin{equation}\label{size}
l=a^{n\log l}\times p^{n\log L}
\end{equation}

\noindent where $n=rn_r$ and both $n\rightarrow \infty$ and $n_r\rightarrow \infty$ as $r\rightarrow \infty$, so that in the limit $r\rightarrow \infty$ we have $L^{-1}=l^{1/s}$, when the limit Cantor set has the Box (Hausdorff) dimension $s=\underset{n_r\rightarrow \infty}{\lim}\frac{\log 2^{n_r}}{\log a^{-n_r}}$ and $a$ is restricted to $0<a<1/2$ to satisfy  the open set condition and $p=2$. The limit $r\rightarrow \infty$ actually normalizes the arbitrarily small $l$ to 1 i.e., $l^{1/r}\rightarrow 1$, so as to allow retrieval of the original measure $l$ at an exponentially smaller level and that too in the logarithmic scale.

For a $p>a^{-1}$, we can choose a Cantor set with scale factor $0<b<1/2$ so that $p=a^{-\frac{\log b^{-1}}{\log 2}}$, so that the desired differential measure follows when $a$ is replaced by $b$.
\end{proof}

\begin{remark}
{\rm An arbitrarily small real variable $x\rightarrow 0$ now has the deformed
value $X=xe^{{\frac{1}{s}}\log (x/\tilde x)} \in \mathcal R$. As a consequence, the precise
extension of the ordinary real line (in a neighbourhood of 0) is given by $X-x=(v(\tilde x)/s)x\log
x^{-1}$, when we set $x/\tilde x=x^{-v(\tilde x(x))}, $ where $v$ is the
ultrametric valuation for infinitesimals $\tilde x\in \bf 0$. Now, infinitesimals $%
\tilde x$ live in the gaps of a Cantor set in (0,$\delta)$ and $v$ when
extended over the closure of (0,$\delta)$ is equivalent to the Cantor
function, having constant values $\alpha_i$ on the connected subintervals
of countable number of gaps. As a consequence, for infinitesimals from any
such gaps we have $X-x=(\alpha_i/s)x\log x^{-1}$, and hence as the real
variable $x$ goes to zero ultimately (that is becomes $O(\delta)$ and
smaller), $X$ still remains nonzero because of logarithmic correction
factor. This new sublinear effect \cite{dp5} is at the heart of the proof of the prime number
theorem that we sketch below.}
\end{remark}

\section{Prime Number Theorem}

We begin by recalling a few salient features of infinitesimals and the associated influences on $R$.

1. Infinitesimals are real numbers in a limiting sense that these are
elements of (0,$\delta)$ satisfying the inversion rule which is valid even
in the limit $\delta\rightarrow 0^+$.

2. Scale invariant infinitesimals $\tilde{X}=\tilde{x}/\delta ,\ \delta
\rightarrow 0$ can not be a constant, i.e. $|\tilde{X}|=\lambda_0 <1$, for
then the corresponding valuation would vanish $v(\tilde{x})=\lim \log
_{\delta }\lambda_0 ^{-1}=0$, thereby reducing
the set of infinitesimals essentially to the trivial infinitesimal 0.

3. The nontrivial form of scale invariant infinitesimals are $\tilde
X=\lambda \delta^{v(\tilde x)}$ (which is not a constant) and so depends on
the choice of scale $\delta$, when the limit in the definition of $v$ is
evaluated with a slight modification of the original definition

\begin{equation}
v(\tilde x):= \underset{n\rightarrow \infty}{\lim}\log_{\delta^{-n}} {\tilde
X}^{-1}, \ \tilde x=\delta^n \tilde X \ \in \mathbf{0}
\end{equation}

The scale invariant infinitesimals are therefore also said to be scale
invariant $\delta -$ infinitesimals \cite{dp3}. The transition between two $\delta $
infinitesimals $\tilde{X}$ and $\tilde{X}^{\prime }$ would be mediated by
inversions (jumps) of the form $\tilde{X}\rightarrow \tilde{X}^{\prime }={%
\tilde{X}}^{-\alpha },$ where $\alpha $ takes values from $\{\mathrm{range}%
\}v(\tilde{x})=N$.

4. The valuation $v$ defined on the set of infinitesimals is a nontrivial
ultrametric. Since the real number system can not have any nontrivial metric
other than the usual one, the set of $\delta-$ infinitesimals must be a
{\em non-real} ultrametric space. Now, given the rational number set $Q$, there
exists inequivalent $p$-adic local fields $Q_p$, for each prime (finite) $p$,
together with $Q_{\infty}\equiv R$. Next, as a consequence of the Frobenius theorem any
ultrametric extension of $R$ must be infinite dimensional. Hence, $\mathbf{R}
$ is an infinite dimensional metrically complete ultrametric space which locally has the product
structure ${\mathbf{0}}= \prod Q_p$ (c.f. Theorem 1) \cite{dp3}. However, it is also well-known that ultrametric spaces has an {\em hierarchical} structure \cite{tree}. Each fibre $Q_p$ is, therefore, called a branch or leaf of the ultrametric space and a corresponding $p$
infinitesimal $\tilde X_p$ (say) living in the $p$th branch $Q_p$ would
transit to the immediate successor $\tilde X_q$ living in the $Q_q$ branch
for the prime $q$ being the next prime exceeding $p$ via inversion induced
jumps.

5. As a real variable $x>0$ approaches 0, because of continuity, it is
supposed to vanish ultimately. However, because of the scale invariant
infinitesimals, it is, however, replaced by the deformed $X=x+(v(\tilde
x)/s)x \log x^{-1}$, living in $\mathcal R$, which is exact for sufficiently small values of $x$.
The first term $x$ is the usual {\em linear} variable which goes to zero in the
limit. The second is a nontrivial correction, which we investigate more in
detail below. Notice that as the dominant term $x$ goes to zero in the limit $x\rightarrow 0 (\sim \delta)$, the {\em motion (variation)} in the sub-dominant correction term is {\em transmitted} to $v(\tilde x)$, when the factor $x\log x^{-1}\sim O(\delta\log \delta^{-1})$ now acts as an {\em infinitesimal} scaling factor. To emphasize, this is a nontrivial real valued  infinitesimal living in the archimedean real number system.

6. First, we recall that the ultrametric norm quantifies the degree of
nontrivial \emph{motion} experienced by a real variable $x$ in presence of
nontrivial scale invariant infinitesimals in the zero set ${\mathbf{0}}=\prod
{Q_p}$. The uniform rate 1 of approach toward 0 of the original variable $x$
gets flickered because of multiple interjections from scale invariant
infinitesimals $\tilde X_p \in Z_p\subset Q_p$, the ring of $p$-adic
integers.

7. As the variable $x$ becomes vanishingly small (so that the first term of $X$
gets subdominant compared to the second), its {\em linear} motion gets transferred
to the scale invariant ultrametric variable $\tilde X_2$ living in $Z_2\subset Q_2$.
However, \emph{the nonvanishing contributions to the deformed real variable}
$X$ from the scale invariant $\tilde X_2$ is received in the form of {\em real
values} via the norm $v$, viz, $X=v(\tilde x_2)\delta\log \delta^{-1/s}$, where we
neglect $x$ and $\tilde x_2=\delta2^{n}\tilde X_2\in Z_2, \ n\rightarrow \infty$.

8. Recall that $v$ is a locally constant function, $dv/dx=0$, but its
variability would get revealed in a growing real variable $x_2$ in the
neighbourhood of $x$: existence of $x_2(x)$ is assured by scale invariant
infinitesimals. Recall $v(\tilde x_2)=\log _{\delta^{-n}}\tilde X_2^{-1}$ and $%
\tilde X_2=\delta^{n\alpha\log (x/x_2)}$, as $n\rightarrow \infty$, so that $%
v(\tilde x_2)=\alpha\log (x/x_2)$, establishing that $dv/dx=0$, but $%
ydv/dy=v, \ y=(x/x_2)$, since $xd x_2/dx=x_2$. Notice that as $%
x\rightarrow 0^+$, in the neighbourhood of each such $x\neq 0$, $x_2
\rightarrow x$, in such a manner that $x/x_2= 1+\eta_2 $, where the scale
invariant $\eta_2$ is a growing variable from 0 in (0,1). As a consequence,
the deformed $X$ would initially have the approximate form $X\approx
(\alpha/s)\eta_2\delta\log \delta^{-1}$, which grows with $\eta_2$ until $\eta_2$
becomes of the order of $O(1)$, viz, $\eta_2=1-\eta_3; \ \eta_3\approx 0$ is
another scale invariant variable living in (0,1), when the motion of $\eta_2$
is transferred to $\eta_3$ by inversion $O(1)\sim \eta_2 \rightarrow
(1-\eta_2)^{-1}=1+\eta_3$ so as to induce a slightly altered value of the
deformed variable $X =(v(\tilde x_3)/s)\delta\log \delta^{-1}=(\alpha/s)(1+\eta_3)\delta\log \delta^{-1}$, and the initially negligible $\eta_3$ grows continuously to $O(1)$ in the interval (0,1). Notice that $\tilde x_3\in Z_3$ and $x/x_3=1+\eta_3$.

9. The above general structure will replicate each time the scale invariant
infinitesimal $\tilde X_p$ living in $Z_p$ switches over to an $\tilde X_q$
in $Z_q$, for the prime $q$ immediate successor of $p$, thus leaving an
imprint over the deformed real $X$ in the form $X=(v(\tilde x_q)/s)\delta\log \delta^{-1}=(\alpha/s)(1+1+\ldots +1+\eta_q)\delta\log \delta^{-1}$, where the sum of 1's
is up to $p$th prime and $\eta_q$ grows linearly from 0 up to 1 in (0,1). The scale
invariant ultrametric infinitesimals $\tilde X_p$ and $\tilde X_q$ switches
between them by inversions. Analogous switching between scale invariant
deformed real variables $x_p=1-\eta_p$ and $x_q=1+\eta_q$ is also mediated
by inversion $x_q=x_p^{-1}$.

10. As a consequence, as $x\rightarrow 0$ more and more \emph{accurately}, the
scale invariant components $x_p$ of the deformed extension $X$ continue to grow, as the corresponding scale invariant infinitesimals continue to evolve over various branches of $Z_p$'s as $p\rightarrow \infty$, until the
asymptotic formula for the prime counting function $\Pi(x^{-1})=1 +1 +\ldots
+ 1 +\eta_q$ becomes sufficiently large so as to cancel exactly the
corrected (real) variable (now acting as such as a scale factor) $\delta\log \delta^{-1}$. This already completes the proof of the PNT.

11. As an aside, let us now observe the following:

\begin{proposition}
The prime counting function $\Pi (x)$ is a locally constant function.
\end{proposition}

To prove this, let us first consider the step function

\begin{equation}\label{lcf}
f(x) = 
   \begin{cases}
   a, &  0 < x < p, \\
    b, & x >p
    \end{cases}
\end{equation}

\noindent with a finite discontinuity at $x=p$ in the usual sense. In the present scale invariant formalism with inversion mode for increments, we now show that $f$ solves $x\frac{df}{dx}=0$ every where, that is, even at $x=p$. As $x$ increases toward $p$ from the left linearly,  the graph of $f$ is a straight line parallel to the $x$-axis. In the left neighbourhood of $p$, $x=p-\eta= px_-, \ x_-=1-\eta/p$. Analogously, in the right neighbourhood, we have $x=p+\tilde\eta= px_+, \ x_+=1+\tilde\eta/p$, so that $x_+=x_-^{-1}$. Let us assume that $\eta$ and $\tilde\eta$ are sufficiently small, so that the point set $\{p\}$ is identified with the closed interval $I_p=[1-\eta/p, 1+\tilde\eta/p]$, and so defines the accuracy  level of a given computational problem.   The interval $I_p$ corresponds to an infinitesimally small neighbourhood of $p$. At the level of this infinitesimal scale, the function $f$ is interpolated by the scale invariant formula 
\begin{equation}\label{lcf2}
\tilde f(x) = 
     \begin{cases} 
         a, & 0<x<px_-, \\
              a +(b-a)\phi_p(x), & x\in pI_p, \\ 
              b, &  px_+ <x.
              \end{cases}
\end{equation}

\noindent where $x=(p-\eta)+(\tilde \eta+\eta)\tilde x, \ 0\leq \tilde x\leq 1$. Clearly, $\tilde f(x)=f(x)$, in the limit $\eta, \ \tilde\eta \rightarrow 0$. Moreover, $x\frac{d \tilde f}{dx}=0$ everywhere, including $x=p$, since the locally constant Cantor function $\phi_p(x)$ on $I_p$ does. It follows, therefore, that as $x$ approaches to $p$ from left and arrives at a point of the form $x=px_-$, it switches smoothly to $x=px_+$ at the right of $p$ by inversion $x_-^{-1}=x_+$. The associated value of the function $f$ i.e. $a$,  however, changes over to $b$ by a cascade of smaller scale self similar smooth jumps as represented by the Cantor function $\phi_p(x)$. The cost of this smoothness, however, is the arbitrariness in the formalism that is introduced via the arbitrariness of the choice of the Cantor function.

The prime counting function $\Pi(x)$ is a step function in the neighbourhood of every prime. Hence the result. $\Box$

As  another interesting application of the present formalism, let us re-examine the rate of divergences of the well known harmonic series and the series of reciprocals of primes. 

Let $H_n=\underset{1}{\overset{n}{\sum}}{\frac{1}{i}}$, and $P_n= \underset{1}{\overset{p_n} {\sum}}{\frac{1}{p}}$, be the $n$th partial sums of the harmonic series and that of the reciprocals of primes respectively. Using a continuous variable $x>0$, we get two infinite ladder functions $H(x)=\underset{i\leq x}{\sum}{\frac{1}{i}}$ and $P(x)=\underset{p \leq x} {\sum}{\frac{1}{p}}$. It is well known that $H(x) \sim \log x$ and $P(x) \sim \log\log x$, for sufficiently large $x\rightarrow \infty$. A basic difference between two infinite step functions is that $H(x)$ grows with uniform step (gap) size 1, when that for $P(x)$ is {\em nonuniform} and gap size grows slowly with primes. 

To study the variability of $H(x)$ as $x\rightarrow \infty$, let us notice trivially that ${\frac{dH}{dx}}=0$, for $n-1<x<n$. But, in the neighbourhood of $x=n$ (say), $H(x - h)-H(x)=n^{-1}, \ h>0$. As $n$ approaches $\infty$ uniformly, the scale invariant real variable $n^{-1}$ gets a deformation factor ${\cal X}_+ = (n^{-1})^{-l}\approx hx^l$ for an arbitrary $l>0$ ($h$ here stands for the infinitesimal scaling of Lemma 7 and $l$ to the uniform jump rate, c.f. Remark 5), so that we have $\Delta H(x)\sim hx^{-1 +l}\sim \Delta x^l$. But, since $l>0$ may be arbitrarily small, $x^l\sim \log x$, as $x\rightarrow \infty$, thus leading to the desired rate $H(x) \sim \log x$.

For the series of primes, uniform inversion induced jumps now tell that, as $x\rightarrow \infty$, the ladder function $P(x)$ more and more behave as a Cantor function and hence asymptotically must vary as double logarithms $P(x) \sim \log\log x$ (Lemma 7, Theorem 3). This seemingly universal behaviour is likely to be present in any variations purely over primes and induced by inversion mediated smooth jumps.

Estimation of Euler and Martens constants $\gamma$ and $\rho$ respectively require more work, and is left for future. 

\section{Error term}

 Now, returning to the PNT, we note that the  inversion mediated motion that is being treated here of the deformed 
scale invariant variable $X$ because of the motion of the scale invariant infinitesimals over various branches of the ultrametric space $\mathbf{0}$ may be considered to be the \emph{internal motion}, relative to the original \emph{external}
motion of the real variable $x$ as it approaches  0 linearly.  Further, {\em  there are two modes of changes available to a scale invariant $\eta \lessapprox
1 $: (i) local mode, when infinitesimals get contributions from infinitely
large scales, defined by $\eta\rightarrow \eta^{\prime}=(1+\tilde \eta)^{-1} (< 1)$
and (ii) global mode, $\eta\rightarrow \eta^{\prime}=\eta^{-1}=1+\tilde \eta (>1)$, where $%
\tilde \eta$ is another growing variable living in a branch immediate
successor to that of $\eta$}.  Next, we observe that in the scale invariant formalism 1 is replaced by $\bf 1$, which asymptotically has the representation ${\bf 1}=\tau (1-\delta^{1/\tau})$, where $\tau= (1+\epsilon\eta^{\prime})$, $\epsilon$ being an appropriate scale factor and $\eta^{\prime}$ is a slowly varying growing variable in (0,1). Translating over in $R$, this equality must be interpreted as $x_-x_+\approx 1$ close to 1 ($x_-\lessapprox 1, \ x_+\gtrapprox 1$).  We also notice that in the present formalism there arise various scales in connection with an arbitrarily small $x$: the primary scale $\delta_0<x$, and associated secondary $p$-adic scales $p^{-n}$, so that the composite scale that becomes relevant in the definition of the nontrivial value $v$ is given as $\delta=\delta_0p^{-n}, \ n\rightarrow \infty$. As the corresponding inverted ultrametric variable $\tilde x$ varies over local $p$-adic fields, thereby accumulating extra nontrivial additive corrections along the growing mode, there also arises another nontrivial scale, we call it as the ``infinitesimal scaling factor", $\epsilon(\delta)=\delta\log \delta^{-1}$ leading to a realization of the PNT, as indicated above, via a cancellation of the two factors as $p\rightarrow \infty$.

As a consequence, the scale invariant extension of an arbitrarily small $x$ (c.f. Remark 6) can be rewritten as 

\begin{equation}\label{extn}
X-x=\epsilon(x) ( x^{\prime}/x) (1-x^{v(x/\tilde x)})
\end{equation}

\noindent where $\epsilon(x)=x\log x^{-1} \ (\sim O(\delta\log \delta^{-1}))$ acts as the nontrivial ``infinitesimal" scaling factor,  the factor 
$x^{\prime}/x=1+\eta$ corresponds to a growing variable following the above growing mode, $x^{\prime}$ being a reparametrized variable in the neighbourhood of $x$ and the 
remaining exponential term $x^{v}$ is a function of the slower  variable $\tau=\log (x/x^{\prime})\approx \epsilon(x) (x^{\prime}/x)$. Notice also that $v=1/\tau= 1/(1+\epsilon\eta^{\prime})$, where $\eta^{\prime}$ is another very slowly growing variable (actually the prime counting function $\Pi(x^{-1}))$.

Taking into consideration the above scaling effects the PNT now assumes the form

\begin{equation}\label{cor}
X\approx (1-x^{\frac{1}{1+\epsilon\eta^{\prime}}})(1+1+\ldots +1+\eta_q)x\log x^{-1}=
(1-x^{\frac{1}{1+\epsilon\Pi(x^{-1})}})\Pi(x^{-1})\epsilon(x)
\end{equation}

\noindent so that the final form of the deformed scale invariant $X$ has the value $X=1$, in the limit $ x\rightarrow 0$ when the PNT with
correction term has the form
 
\begin{equation}\label{cor2}
\Pi(x^{-1})x\log x^{-1}= (1-x^{\frac{1}{1+\epsilon\Pi(x^{-1})}}]^{-1} 
= 1+O(x^{\frac{1}{2}-\epsilon^{\prime}})
\end{equation} 

\noindent for an $\epsilon^{\prime}>0$.  As $x\rightarrow 0$, classically (i.e., becomes of O($\delta)$ and less in a computational problem with accuracy $\delta$), the scale invariant deformed variable $X$ continues to evolve, following the above pattern, as the associated dynamical infinitesimal $\tilde X_p$ varies over various local fields $Q_p$, thus contributing successively a unity for every $p$. This process of {\em internal} evolution in the structure of the deformed scale invariant variable $X$ continues until the prime counting function cancels the infinitesimal scaling factor $\epsilon$, in both the growing as well as the local modes of eq(\ref{cor}), in a manner consistent with eq(\ref{cor2}). Assuming $\Pi(x^{-1})x\log x^{-1}=1+\sigma$, the consistency means the equality $\sigma= x^{\frac{1}{2+\sigma}}$, for an arbitrarily small $\sigma(x)$. Clearly, eq(\ref{cor2}) realizes precisely the error term dictated by the Riemann Hypothesis, and hence constitues an indirect proof of the same. One may interprete this error term as a reflection of the presence of randomness in the variations of scale invariant components of the form $\mathcal X$ of a deformed real variable in the neighbourhood of 1, as such a scale invariant variable lives in and varies over a Cantor set. However, this randomness is, of course, not a pure white noise, as it is evident from the nonzero $\epsilon^{\prime}>0$.

In the classical analysis $X=0$, in the limit $x\rightarrow 0$. The counterintuitive growth of the value of $X$ in the scale invariant formalism may be explained in the light of invariance and the associated growth of measure in Sec. 4. The exact real numbers of classical analysis with fixed (Euclidean/Archimedean) values are replaced in the scale invariant formalism by a collection of infinitesimally small neighbours (as, may be recalled, ordinary 0 is replaced by a Cantor set of soft zeros), with nontrivial effects on the value of a real number, which, as shown above, is likely to become activated in a limiting process involving many infinitely large scales. Such a limiting process is abundantly available in complex systems, for instance, in many body systems in astronomy and astrophysics, in biological systems, in finance and similar others fields. The standard classical analytic results are, nevertheless, expected to be valid in a process (context) involving moderate scales ($\delta \sim O(1)$). To cite an example of growth of value, let us recall the variability of the {\em value} of money: money left in a locker at home over a  few days or weeks keeps its value, but put in a bank or in (protected) share market, that is, when money is made to ``flow", gains in value. The concept of limit in classical analysis is {\em passive}, that is, $x\rightarrow 0$ means $x=0$ (recall that uniform (in the usual topology) motion is equivalent to ``no" motion in $R$). The nonarchimedean valuation of soft zeros, however, make the uniform motion of $x\rightarrow 0$ nonuniform (as stated in Remark 5, this nonuniformity includes both uniform and nonuniform variations in ultrametric topology). As a consequence, the scale invariant extension of $x$, viz, ${\cal X}=X/x$ assumes a {\em directed evolutionary} character, thereby driving the limiting value of $X$ not only to 1, but in the process also yields a novel proof of the PNT with correction term that follows from the Riemann Hypothesis. However, 1 in the scale invariant formalism is never exactly 1, but the fattened $\bf 1$, so that we now actually have $X=1\approx (1-x^{\frac{1}{\tau_1}})\tau_1$, where $\tau_1$ is a higher order  O(1) slowly growing variable involving higher order prime counting function for the logarithmic variable $\log x^{-1}$, and so on.  

\section{Final Remarks}

The proof of the prime number theorem and Riemann's hypothesis as well as the concept of  novel invariance properties of measure and cardinality considered here are some what counterintuitive in the standard, conventional approach. However, these seem to follow quite naturally and elegantly from the so called internal motion of an inverted variable in an infinite dimensional ultrametric space accommodating dynamic infinitesimals. The ordinary singleton set of  0 is fractured into an ultrametric Cantor set, transitions between points of which are accomplished by inversion induced smooth jumps. Ordinary jump processes in probability and stochastic processes are nonsmooth and require a lot more technical details. The smooth jumps are equivalently interpreted  as double logarithmic variations of an ordinary real variable. The associated deformed real number system $\mathcal R$ accordingly has the structure of  a positive measure Cantor set, thus endowing the linear motion of $x\rightarrow 0$ in $R$ with a nontrivial sublinear flow $X-x \approx x\log x^{-1},$ as $x\rightarrow 0$. Before closing, we note that the possible connection of a first order ODE of the form (\ref{sfe}) with Riemann's hypothesis appears to have been first pointed out in Ref.\cite{berry} in the context of quntum mechanics.

\end{document}